\documentclass[final,hidelinks,onefignum,onetabnum]{siamart251216}
\usepackage{amsmath,amssymb}
\usepackage{booktabs}
\usepackage{float}

\newtheorem{assumption}{Assumption}
\newtheorem{remark}{Remark}

\headers{Positivity for FEM Richards Equation}{Benfanich, Bourgault, and Beljadid}
\title{Discrete positivity and maximum principles for a finite element discretization of the Richards equation}
\author{
Abderrahmane Benfanich\thanks{Department of Mathematics and Statistics, University of Ottawa, Ottawa, ON K1N~6N5, Canada (\email{abenf099@uottawa.ca}). Corresponding author.}
\and Yves Bourgault\thanks{Department of Mathematics and Statistics, University of Ottawa, Ottawa, ON K1N~6N5, Canada (\email{ybourg@uottawa.ca}). This work was supported by NSERC Discovery Grant RGPIN-2019-06855.}
\and Abdelaziz Beljadid\thanks{Department of Mathematics and Statistics, University of Ottawa, Ottawa, ON K1N~6N5, Canada, and University Mohammed VI Polytechnic, Benguerir 43150, Morocco (\email{Abdelaziz.BELJADID@um6p.ma}, \email{abeljadi@uottawa.ca}). This work was supported by NSERC Discovery Grant RGPIN/5220-2022 and NSERC DGECR/526-2022.}
}

\begin{document}

\maketitle

\begin{abstract}
Standard finite element discretizations of the Richards equation may violate the discrete minimum principle, producing unphysical negative saturations. While existing bound-preserving methods typically rely on computationally expensive fully implicit solvers, we propose a novel semi-implicit finite element framework utilizing a bounded continuous auxiliary variable. Our approach treats the gravity-driven advective term using a linearly implicit technique, which improves the time-step restrictions required by explicit gravity methods near the degenerate limit. We provide rigorous mathematical proofs establishing sufficient geometric and algebraic constraints for discrete positivity and the discrete maximum principle, specifically a local P\'eclet condition and a discrete row-sum condition. When both conditions are satisfied on weakly acute meshes with mass lumping, our framework ensures that numerical solutions strictly respect physical bounds across highly degenerate conditions and initially dry soil regimes. Comprehensive numerical validation demonstrates the method across multiple flow regimes, including cases where algebraic conditions are satisfied, violated, and recovered through mesh refinement.
\end{abstract}



\section{Introduction}
The Richards equation is the standard mathematical model used for simulating water infiltration in unsaturated porous media \cite{richards1931, zha2019}. This degenerate parabolic partial differential equation features a highly non-linear diffusive term and a gravity-driven advective term, taking the continuous form of a quasilinear elliptic-parabolic equation \cite{alt1983}. The fundamental theory of capillary conduction and the continuous model have been thoroughly established in early literature \cite{richards1931}, commonly utilizing the pressure-based, saturation-based, or mixed formulations \cite{celia1990, zha2019}. Each formulation presents its own computational drawbacks \cite{zha2019, farthing2017}. The pressure-based form generally produces significant global mass balance errors unless prohibitively small time steps or mass-conservative spatial/temporal schemes are utilized \cite{celia1990, rathfelder1994, zha2017}, while the saturation-based formulation circumvents these mass balance issues in the dry limit but suffers from unbounded soil-water diffusivity in fully saturated regions \cite{kirkland1992, hills1989, zha2019}. The mixed formulation mitigates many mass conservation issues but still requires careful numerical handling to address its inherent nonlinearities \cite{celia1990, kavetski2001, zha2019}. Notably, Celia et al. \cite{celia1990} were among the first to address the mixed formulation by successfully applying a mass-conservative modified Picard iteration method to solve infiltration problems.

To address the numerical difficulties arising from the degeneracy and non-smoothness of the Richards equation, several stabilization strategies have been developed over the years \cite{zha2019, farthing2017}. Standard iterative solvers using the Newton-Raphson method or standard Picard iterations often fail to converge or require severe time-step restrictions near sharp wetting fronts where derivatives vanish or diverge \cite{list2016, jones2001, lehmann1998, paniconi1991}. Primary variable switching dynamically selects either the pressure head or the saturation as the primary unknown to circumvent these issues \cite{diersch1999, wuforsyth2001, forsyth1995, zha2019}. Under a Newton-Raphson framework, this technique typically takes the derivative with respect to pressure head for saturated nodes and saturation for unsaturated nodes to mitigate extreme non-linearities \cite{krabbenhoft2007, brunner2012}. However, as highlighted in comprehensive reviews \cite{zha2019, farthing2017}, variable switching approaches often suffer from non-smooth, dynamic transitions between primary variables that can lead to physically unrealistic solutions \cite{krabbenhoft2007, zha2019, zha2017}. While improvements such as hybrid and generalized switching criteria have been proposed to minimize these deficiencies \cite{hassane2017, zeng2018}, the performance of these techniques tends to remain highly problem-specific \cite{zha2019}. Another widely used strategy is mathematical regularization, which replaces degenerate constitutive relationships with smoothed approximations, inherently introducing artificial modeling parameters \cite{schweizer2007, popschweizer2011, fevotte2024}. Stabilized fixed-point iterations like the L-scheme and the modified L-scheme guarantee global convergence for degenerate problems without relying strictly on the derivative of the nonlinearity \cite{pop2004, slodicka2002, mitra2019, seus2018}. As a recent alternative to mathematical transformation techniques that might introduce unfavorable discontinuity \cite{zha2019, chen2016}, formulations utilizing bounded continuous auxiliary variables have been proposed to naturally handle both fully saturated and completely dry regimes \cite{benfanich2025, benfanich2026a}. This transformation maps potentially unbounded physical variables into a strictly bounded domain to eliminate unbounded terms from the governing equations, effectively removing the need for mathematical regularization \cite{benfanich2025}.

However, a persistent and significant challenge for numerical methods of the Richards equation is maintaining intrinsically the physical bounds of the solution \cite{barrenechea2024, zha2019}. Physical consistency requires that numerical approximations strictly respect both non-negativity and maximum saturation bounds \cite{misiats2013, svyatskiy2017}. The mathematical formulation of a numerical method's ability to respect these physical bounds is known as the discrete maximum principle \cite{barrenechea2024, karatson2007}. In many applications where convection or gravity dominates diffusion, standard direct finite element and finite volume discretizations fail to satisfy the discrete maximum principle \cite{barrenechea2024, svyatskiy2017}. To strictly respect physical bounds, the spatial and temporal discretizations must satisfy specific algebraic constraints \cite{barrenechea2024, farago2006}. For linear finite element methods, the discrete system matrix must generally be of nonnegative type or an M-matrix, which mathematically requires non-positive off-diagonal entries and non-negative row sums \cite{berman1994, varga2000, stoyan1986}. Geometrically, this requires evaluating the spatial discretization on weakly acute or Xu-Zikatanov meshes \cite{barrenechea2024, xu1999}. Furthermore, mass lumping is an essential integration technique necessary to discretize transient parabolic problems; it diagonalizes the mass matrix through row summation and prevents positive off-diagonal entries from violating the discrete maximum principle \cite{thomee2006, celia1990}.

Because standard first-order bound-preserving schemes severely degrade spatial accuracy and smear infiltration fronts, researchers have developed algebraic flux correction and flux-corrected transport methodologies \cite{forsyth1997, oulhaj2018, misiats2013}. In the finite volume context, second-order accurate monotone schemes utilizing multi-point flux approximations and nonlinear solution-dependent stencils have been proposed, strictly satisfying the discrete maximum principle for the Richards equation on unstructured meshes \cite{misiats2013, svyatskiy2017, lipnikov2012}. Similarly, originally introduced for finite difference methods \cite{boris1973} and later extended to multidimensional unstructured meshes \cite{zalesak1979, kuzmin2012}, flux-corrected transport schemes blend the bound-preserving properties of low-order predictors with the high-resolution accuracy of high-order methods by limiting antidiffusive fluxes using multidimensional limiters \cite{boris1973, zalesak1979, kuzmin2012}. A recent study has successfully extended these flux-corrected transport schemes to the nonlinear, degenerate parabolic structure of the Richards equation, demonstrating optimal second-order convergence on unstructured meshes while strictly preserving the physical bounds of the low-order scheme \cite{barua2026}.

In this paper, we explicitly address the limitations of current bound-preserving schemes by establishing novel frameworks for semi-implicit time discretizations. While previous literature successfully enforces bounds via mass lumping, upwinding techniques, and high-order flux limiting, rigorous theoretical proofs of nonlinear stability, positivity, and the discrete maximum principle have been established almost exclusively for fully implicit time discretizations \cite{forsyth1997, oulhaj2018, barua2026, svyatskiy2017}. Fully implicit methods are unconditionally stable but require computationally expensive nonlinear solvers at each time step \cite{shahraiyni2012, kavetski2002, zha2019}. In contrast, our paper studies semi-implicit methods to solve the Richards equation utilizing the bounded auxiliary variable approach \cite{benfanich2025, keita2021}. A key novelty of this work is the introduction of a method that treats the gravity term linearly implicitly. While explicit gravity method treatments of advection require severe critical time-step restrictions to maintain non-negativity near the degenerate limit \cite{hills1989, kavetski2002}, treating gravity linearly implicitly significantly weakens the condition on the time step. We rigorously establish sufficient geometric and algebraic conditions, specifically local Péclet and row-sum conditions \cite{barrenechea2024}, to guarantee discrete maximum principle and positivity for this linearly implicit formulation without relying on fully implicit solvers. Comprehensive numerical validation across diffusion-dominated, mixed, and advection-dominated flow regimes demonstrates the sufficiency of the theoretical conditions.

The remainder of this article is structured as follows. Section~\ref{sec:model} introduces the continuous Richards equation with initial and boundary conditions. Section~\ref{sec:discretization} presents the complete discretization framework, including temporal and spatial discretizations, finite element spaces, mass lumping integration, and the geometric and algebraic properties required for maximum principles on weakly acute meshes. Section~\ref{sec:analysis} develops the theoretical analysis to establish discrete positivity and maximum principle preservation, progressing from the explicit gravity method through the linearly implicit advection scheme to the sufficient conditions (Péclet and row-sum) that guarantee discrete bounds. Finally, Section~\ref{sec:numerical} provides comprehensive numerical validation across multiple flow regimes: (1) diffusion-dominated scenarios where conditions hold and bounds are preserved, (2) mixed regimes with moderate violations, and (3) mesh refinement studies demonstrating recovery of bounds through spatial resolution.

\section{Richards' equation}\label{sec:model}
The Richards equation \cite{richards1931} is the standard model for unsaturated flow in porous media. Written in its mixed formulation \cite{alt1983, celia1990}, and expressed using a bounded auxiliary variable $u$ \cite{benfanich2025, benfanich2026a}, it takes the form:
\begin{equation}
\partial_t \theta(u) - \nabla \cdot (K(u) \nabla u) - \nabla \cdot (\overline{K}(u) \mathbf{e}_z) = 0 \quad \text{in } \Omega \times (0,T]
\end{equation}
Here, $\theta(u) \in [0,1]$ is the effective saturation, a monotonically increasing function of the auxiliary variable $u$, and $K(u) \ge 0$ and $\overline{K}(u) \ge 0$ are the diffusive and advective hydraulic conductivity coefficients, respectively.
Following \cite{benfanich2026a}, we extend the definition of the nonlinear functions to the entire real line by setting $\theta(u) = u$ for $u < 0$ and for the remainder of the real line, by symmetry about the point of saturation $(u^*, 1)$ where $\theta(u^*) = 1$. We extend the function $K(u)$ to the entire real line by setting $K(u) = K(u^*)$ for $u > u^*$ and $K(u) = K(-u)$ for $u < 0$. We extend the function $\overline{K}$ similarly.
\subsection{Initial and boundary conditions}\label{sec:ic-bc}
For the bounded domain $\Omega \subset \mathbb{R}^d$ with boundary $\partial \Omega$, the continuous problem requires the following conditions:
\begin{itemize}
  \item \textbf{Initial condition:} At time $t=0$, an initial profile $u_0(\mathbf{x})$ is prescribed for $\mathbf{x} \in \Omega$.
  \item \textbf{Boundary conditions:} Dirichlet boundary conditions are imposed on the entire boundary $\partial \Omega$.
\end{itemize}
In our numerical framework, the discrete boundary data $u_{b,h}$ is the finite element extension of the continuous boundary data $u_b$, with positive values on boundary nodes and zero values on interior nodes.

\section{Discretization framework}\label{sec:discretization}

\subsection{Time discretization}\label{sec:time}
Let $[0,T]$ be the time interval. We define a temporal partition
 $0=t_0<t_1<\cdots<t_N=T$. For dynamic time stepping, we set
 $\tau_n:=t_n-t_{n-1}>0$ for $n=1,\ldots,N$, and equivalently
 $t_n=\sum_{k=1}^{n}\tau_k$ for $n=1,\ldots,N$.
The numerical approximation of the solution at time $t_n$ is denoted by $u^n$.

\subsection{Spatial discretization}\label{sec:spatial}
Let the domain $\Omega \subset \mathbb{R}^d$ be a polyhedron represented by a conforming, shape-regular simplicial mesh $\mathcal{T}_h$.
\begin{itemize}
    \item Let $T \in \mathcal{T}_h$ denote an individual element.
    \item Let $h_T$ be the diameter of element $T$, and $h = \max_{T \in \mathcal{T}_h} h_T$ be the global mesh size parameter.
    \item Let $\mathcal{N}_h = \{x_i\}_{i=1}^{N_h}$ denote the set of all interior nodal vertices in the mesh $\mathcal{T}_h$.
\end{itemize}

\subsection{Finite element space}\label{sec:fes}
We define the standard conforming finite element space $V_h \subset \mathcal{H} = H^1_0(\Omega)$ consisting of continuous, piecewise linear functions ($\mathbb{P}_1$ elements):
$$ V_h = \{ v_h \in C^0(\overline{\Omega}) : v_h|_T \in \mathbb{P}_1(T) \ \forall T \in \mathcal{T}_h, \ v_h|_{\partial \Omega} = 0 \}. $$
We associate this space with the standard nodal basis functions $\{\phi_i\}_{i=1}^{N_h}$, uniquely defined by the property $\phi_i(x_j) = \delta_{ij}$. 

\subsection{Nodal interpolation operator}\label{sec:interpolation}
Because the nonlinear function $\theta(u_h)$ evaluated on a finite element function does not generally belong to $V_h$, we define the standard nodal interpolation operator $\mathcal{I}_h : C^0(\overline{\Omega}) \to V_h$ as:
$$ \mathcal{I}_h(v)(x) = \sum_{i=1}^{N_h} v(x_i) \phi_i(x). $$

\subsection{Numerical integration using mass lumping}\label{sec:masslumping}
To satisfy the algebraic conditions needed for a discrete maximum principle, as is standard practice for transient parabolic problems \cite{thomee2006}, we introduce the lumped $L^2$ inner product $(\cdot, \cdot)_h$, defined by applying the nodal interpolation operator before integration:
$$ (u, v)_h = \int_\Omega \mathcal{I}_h(u v) \, dx. $$
For two discrete functions $u_h, v_h \in V_h$, the cross terms vanish, yielding:
$$ (u_h, v_h)_h = \sum_{i=1}^{N_h} u_h(x_i) v_h(x_i) \int_\Omega \phi_i(x) \, dx = \sum_{i=1}^{N_h} m_i U_i V_i, $$
where $m_i = \int_\Omega \phi_i(x) \, dx > 0$ is the lumped mass associated with a node $i$, and $U_i, V_i$ represent the corresponding nodal values.

\subsection{Global mesh and matrix properties}\label{sec:mesh}
Before defining the time-stepping schemes, we establish the fundamental geometric properties required for the spatial discretization to satisfy maximum principles.

\begin{assumption}[Weakly acute mesh condition \cite{barrenechea2024}, Def. 2.2 \& Eq. (2.15)]\label{ass:weakly-acute}
We assume that the triangulation $\mathcal{T}_h$ is weakly acute, and equivalently for $\mathbb P_1$ elements the basis functions satisfy:
$$ \nabla \phi_i \cdot \nabla \phi_j \le 0 \quad \text{for all } i \neq j \text{ on every } T \in \mathcal{T}_h. $$
\end{assumption}

For $\mathbb P_1$ elements, gradients are constant on each simplex $T \in \mathcal{T}_h$. Using Eq. (2.15) in \cite{barrenechea2024}, we directly obtain
$$
\nabla\phi_i\cdot\nabla\phi_j = -\frac{|F_i^T|\,|F_j^T|}{d^2|T|^2}\cos\theta_E^T \qquad (i\neq j).
$$
Here, $F_i^T, F_j^T$ are the facets of $T$ opposite vertices $i, j$, $d$ is the spatial dimension, and $\theta_E^T$ is the interior dihedral angle between the facets $F_i^T$ and $F_j^T$.
Since the prefactor is strictly positive, $\nabla\phi_i\cdot\nabla\phi_j\le 0 \iff \cos\theta_E^T\ge 0 \iff \theta_E^T\le \pi/2$, which is the weakly acute condition.

Let the lagged-coefficient stiffness matrix $A^{n-1} = (A_{ij}^{n-1})_{1 \le i, j \le N_h}$ be defined by the entries $A_{ij}^{n-1} = \int_\Omega K(u_h^{n-1}) \nabla \phi_j \cdot \nabla \phi_i \, dx$. This matrix satisfies the following critical property:

\begin{lemma}[Matrix of nonnegative type \cite{barrenechea2024}, Def. 3.2]\label{lem:nonnegative-type}
Since $K(u) \ge 0$, Assumption~\ref{ass:weakly-acute} guarantees that the off-diagonal entries are non-positive: $A_{ij}^{n-1} \le 0$ for all $i \neq j$. Furthermore, the row sums of the matrix are non-negative: $\sum_{j=1}^{N_h} A_{ij}^{n-1} \ge 0$ for all $i$. Together, these two conditions guarantee that $A^{n-1}$ is a matrix of nonnegative type.
\end{lemma}

\section{Discrete positivity and maximum principles}\label{sec:analysis}

\subsection{Fully discrete variational scheme using the explicit gravity method}\label{sec:explicit}
We consider a semi-implicit time discretization where the diffusion term is treated in lagged-coefficient (linearly implicit) form and the advective term is evaluated using the explicit gravity method. At the absolute degenerate limit ($u \to 0$), the function $\theta(u) \to 0$, which forces the explicit gravity method time-step restriction to vanish for nodes adjacent to a homogeneous zero boundary. 

To prevent the allowable time step from vanishing, we impose a strictly positive Dirichlet boundary condition. We define the corresponding affine finite element manifold.

Let $W_h \subset H^1(\Omega)$ be the unconstrained finite element space defined by:
$$ W_h = \{ v_h \in C^0(\overline{\Omega}) : v_h|_T \in \mathbb{P}_1(T) \ \forall T \in \mathcal{T}_h \}. $$
The test space with homogeneous boundary conditions is the standard subspace $V_h = W_h \cap H^1_0(\Omega)$. 

For this section, let the nodal set be partitioned as $I \cup \Gamma$, where $I$ is the set of interior nodes and $\Gamma$ is the set of Dirichlet boundary nodes.

We assume the boundary data $u_b$ is the strictly positive trace of an $H^1(\Omega)$ function, such that $u_b \in H^{1/2}(\partial \Omega)$ and $u_b \ge u_{\min} > 0$ almost everywhere on $\partial \Omega$. We define its discrete finite element extension $u_{b,h} \in W_h$ as:
$$ u_{b,h}(x) = \sum_{j \in \Gamma} u_b(x_j) \phi_j(x), $$
which strictly satisfies $u_{b,h}(x_j) > 0$ on the boundary nodes $\Gamma$ and vanishes on all interior nodes $I$.

\begin{assumption}[Initial and boundary compatibility]\label{ass:init-boundary}
To ensure that the time step $\tau_n$ is well-defined, the boundary condition and the initial discrete state $U^0 \in W_h$ must be compatible and bounded from below by the strictly positive physical constant $u_{\min} > 0$:
$$ u_b(x) \ge u_{\min} > 0 \text{ on } \partial \Omega, \quad \text{and} \quad U_i^0 \ge u_{\min} > 0 \text{ for all } i \in I. $$
\end{assumption}

For a given discrete time $t_n$ and previous solution $u_h^{n-1} \in W_h$ (with $u_h^{n-1}|_{\partial \Omega} = u_{b,h}$), find the discrete solution $u_h^n \in W_h$ restricted to the affine manifold $u_h^n - u_{b,h} \in V_h$ such that for all interior test functions $\phi_i \in V_h$:
$$ \left(\frac{\theta(u_h^n) - \theta(u_h^{n-1})}{\tau_n}, \phi_i\right)_h + \int_\Omega K(u_h^{n-1}) \nabla u_h^n \cdot \nabla \phi_i \, dx + \int_\Omega \overline{K}(u_h^{n-1}) \mathbf{e}_z \cdot \nabla \phi_i \, dx = 0. $$

The finite element solution at time $t_n$ is explicitly decomposed into the sum of unknown interior degrees of freedom and known discrete boundary values:
$$ u_h^n(x) = \sum_{j \in I} U_j^n \phi_j(x) + \sum_{j \in \Gamma} u_{b,h}(x_j) \phi_j(x). $$

Substituting this decomposition into the lagged-coefficient diffusion term and testing against an interior basis function $\phi_i$ ($i \in I$) yields:
$$
\begin{aligned}
\int_\Omega K(u_h^{n-1}) \nabla u_h^n \cdot \nabla \phi_i \, dx
&= \sum_{j \in I} U_j^n \int_\Omega K(u_h^{n-1}) \nabla \phi_j \cdot \nabla \phi_i \, dx \\
&\quad + \sum_{j \in \Gamma} u_{b,h}(x_j) \int_\Omega K(u_h^{n-1}) \nabla \phi_j \cdot \nabla \phi_i \, dx.
\end{aligned}
$$

Using the definition
$$
A_{ij}^{n-1} = \int_\Omega K(u_h^{n-1}) \nabla \phi_j \cdot \nabla \phi_i \, dx,
$$
this gives:
$$
\int_\Omega K(u_h^{n-1}) \nabla u_h^n \cdot \nabla \phi_i \, dx
= \sum_{j \in I} A_{ij}^{n-1} U_j^n + \sum_{j \in \Gamma} A_{ij}^{n-1} u_{b,h}(x_j).
$$

Substituting this expanded diffusion term back into the fully discrete variational equation, applying the lumped mass integration, and isolating the unknown interior terms on the left-hand side gives the algebraic nodal equation for all $i \in I$:
$$
\begin{aligned}
m_i \frac{\theta(U_i^n) - \theta(U_i^{n-1})}{\tau_n} + \sum_{j \in I} A_{ij}^{n-1} U_j^n
&= -\int_\Omega \overline{K}(u_h^{n-1}) \mathbf{e}_z \cdot \nabla \phi_i \, dx \\
&\quad - \sum_{j \in \Gamma} A_{ij}^{n-1} u_{b,h}(x_j).
\end{aligned}
$$

We define the right-hand side as the modified explicit gravity load vector $\widetilde{G}^{n-1} \in \mathbb{R}^{|I|}$, which incorporates both the explicit gravity method advection and the diffusion boundary contribution from the lagged-coefficient matrix:
$$ \widetilde{G}_i^{n-1} := G_i^{n-1} - \sum_{j \in \Gamma} A_{ij}^{n-1} u_{b,h}(x_j), $$
where $G_i^{n-1} = -\int_\Omega \overline{K}(u_h^{n-1}) \mathbf{e}_z \cdot \nabla \phi_i \, dx$. 

Because $i \in I$ is an interior node and $j \in \Gamma$ is a boundary node, we have $i \neq j$. By Lemma~\ref{lem:nonnegative-type}, the off-diagonal entries are non-positive: $A_{ij}^{n-1} \le 0$. Since the discrete boundary condition is strictly positive ($u_{b,h}(x_j) > 0$), the boundary flux contribution on the right-hand side acts as a strictly non-negative source term: $-A_{ij}^{n-1} u_{b,h}(x_j) \ge 0$.

\begin{assumption}[Strict algebraic time-step restriction]\label{ass:strict-timestep}
To guarantee strict positivity of the solution, the advective term in the explicit gravity method must be algebraically bounded. We assume the time step $\tau_n > 0$ satisfies the following strict inequality for all interior nodes $i \in I$ where $\widetilde{G}_i^{n-1} < 0$:
$$ \tau_n < \frac{m_i \theta(U_i^{n-1})}{|\widetilde{G}_i^{n-1}|}. $$
\textbf{Well-Definedness of $\tau_n > 0$:} By the induction hypothesis, $U_i^{n-1} > 0$. Since the function $\theta$ is strictly increasing and $\theta(0) = 0$, $\theta(U_i^{n-1}) > 0$. Because the lumped mass $m_i > 0$, the fraction defining the upper bound is strictly positive. As the mesh contains a finite number of interior nodes, taking the minimum of these bounds ensures a well-defined time step $\tau_n > 0$.
\end{assumption}

\begin{remark}[Vanishing time step at the degenerate limit]
If we had maintained homogeneous Dirichlet boundary conditions ($u_{b,h} = 0$), the nodes adjacent to the boundary would naturally satisfy $U_i^{n-1} \to 0$, implying $\theta(U_i^{n-1}) \to 0$ since $\theta(0) = 0$. If the local advective flux is strictly negative, this condition would force the allowable time step $\tau_n \to 0$, essentially halting the simulation. The strictly positive boundary condition ensures that $U_i^{n-1}$ and $\theta(U_i^{n-1})$ remain bounded away from zero, guaranteeing the existence of a positive time step $\tau_n > 0$.
\end{remark}

\begin{theorem}[Discrete strict positivity principle]\label{thm:strict-positivity}
Under the weakly acute mesh condition and the strict time-step restriction, if the initial interior discrete state is strictly positive ($U_i^0 > 0$ for all $i \in I$), then the discrete solution at any time step $n$ remains strictly positive: $U_i^n > 0$.
\end{theorem}

\begin{proof}
We proceed by mathematical induction. Assume that $U_i^{n-1} > 0$ for all $i \in I$. We employ a discrete variational energy method based on Stampacchia truncation \cite{kinderlehrer1980}. Define the negative part of the interior discrete solution as the vector $v \in \mathbb{R}^{|I|}$, with components given by $v_i = \min(0, U_i^n)$. By definition, $v_i \le 0$, and we can decompose the solution as $U_j^n = (U_j^n)^+ + v_j$, where $(U_j^n)^+ = \max(0, U_j^n) \ge 0$.

Taking the Euclidean inner product of the algebraic nodal balance equation with $v$ yields:
$$ \sum_{i \in I} m_i (\theta(U_i^n) - \theta(U_i^{n-1})) v_i + \tau_n \sum_{i \in I} \sum_{j \in I} A_{ij}^{n-1} U_j^n v_i = \tau_n \sum_{i \in I} \widetilde{G}_i^{n-1} v_i. $$

\textbf{1. The diffusion term:} Using the decomposition $U_j^n = (U_j^n)^+ + v_j$ and noting that the disjoint supports imply $(U_i^n)^+ v_i = 0$, we expand the bilinear form strictly over the interior nodes:
$$ \sum_{i \in I} \sum_{j \in I} A_{ij}^{n-1} U_j^n v_i = \sum_{i \in I} \sum_{\substack{j \in I \\ j \neq i}} A_{ij}^{n-1} (U_j^n)^+ v_i + \sum_{i \in I} \sum_{j \in I} A_{ij}^{n-1} v_j v_i. $$
By Lemma~\ref{lem:nonnegative-type}, $A_{ij}^{n-1} \le 0$ for $i \neq j$. Since $(U_j^n)^+ \ge 0$ and $v_i \le 0$, the first term is non-negative. For the second term, we algebraically expand the quadratic form:
$$ \sum_{i \in I} \sum_{j \in I} A_{ij}^{n-1} v_j v_i = \sum_{i \in I} \left( \sum_{j \in I} A_{ij}^{n-1} \right) v_i^2 - \frac{1}{2} \sum_{i \in I} \sum_{\substack{j \in I \\ j \neq i}} A_{ij}^{n-1} (v_i - v_j)^2. $$
Using the partition of unity ($\sum_{j \in I \cup \Gamma} \phi_j \equiv 1 \implies \sum_{j \in I \cup \Gamma} \nabla \phi_j = \mathbf{0}$), the row sum of the interior lagged-coefficient stiffness matrix equates to:
$$ \sum_{j \in I} A_{ij}^{n-1} = \int_\Omega K(u_h^{n-1}) \nabla \phi_i \cdot \left( \sum_{j \in I \cup \Gamma} \nabla \phi_j \right) dx - \sum_{j \in \Gamma} A_{ij}^{n-1} = - \sum_{j \in \Gamma} A_{ij}^{n-1}. $$
Because $j \in \Gamma$ and $i \in I$ implies $i \neq j$, Lemma~\ref{lem:nonnegative-type} guarantees $A_{ij}^{n-1} \le 0$. Therefore, $\sum_{j \in I} A_{ij}^{n-1} \ge 0$. Consequently, both components in the quadratic expansion are non-negative, yielding:
$$ \sum_{i \in I} \sum_{j \in I} A_{ij}^{n-1} U_j^n v_i \ge 0. $$

\textbf{2. The evolution and advection terms:} Substituting the diffusion bound back into the balance equation and rearranging to group the terms evaluated at $t_{n-1}$ yields:
$$ \sum_{i \in I} m_i \theta(U_i^n) v_i \le \sum_{i \in I} \left( m_i \theta(U_i^{n-1}) + \tau_n \widetilde{G}_i^{n-1} \right) v_i. $$

\textbf{3. Applying the stability condition:} By the strict time-step restriction, $m_i \theta(U_i^{n-1}) + \tau_n \widetilde{G}_i^{n-1} > 0$. Because $v_i \le 0$, the product of these terms is non-positive. Summing this over all interior nodes implies the right-hand side is $\le 0$, leading to:
$$ \sum_{i \in I} m_i \theta(U_i^n) v_i \le 0. $$

\textbf{4. Non-negativity:} If $v_i < 0$, it implies $U_i^n < 0$. Because $\theta$ is strictly increasing with $\theta(0)=0$, this implies $\theta(U_i^n) < 0$. Since $m_i > 0$, the product is strictly positive: $m_i \theta(U_i^n) v_i > 0$. Therefore, for the sum of non-negative terms to be $\le 0$, every term must be exactly zero. This requires $v_i = 0$ for all nodes, concluding that $U_i^n \ge 0$.

\textbf{5. Strict positivity:} Assume by contradiction that there exists a node $i \in I$ such that $U_i^n = 0$. Since we established $U_j^n \ge 0$ for all $j \in I$, and knowing $\theta(0)=0$, the algebraic nodal equation evaluated at node $i$ simplifies to:
$$ m_i \frac{0 - \theta(U_i^{n-1})}{\tau_n} + A_{ii}^{n-1}(0) + \sum_{\substack{j \in I \\ j \neq i}} A_{ij}^{n-1} U_j^n = \widetilde{G}_i^{n-1}. $$
Rearranging this equation gives:
$$ \frac{m_i \theta(U_i^{n-1})}{\tau_n} + \widetilde{G}_i^{n-1} = \sum_{\substack{j \in I \\ j \neq i}} A_{ij}^{n-1} U_j^n. $$
By Lemma~\ref{lem:nonnegative-type}, $A_{ij}^{n-1} \le 0$ for $i\neq j$. Since $U_j^n \ge 0$, the right-hand side is non-positive ($\le 0$). However, applying the strict inequality from assumption~\ref{ass:strict-timestep} ensures that the left-hand side is strictly positive ($> 0$). This yields a contradiction. Therefore, $U_i^n$ cannot be zero, which proves that $U_i^n > 0$.
\end{proof}

\begin{remark}[Unconditional positivity in favorable flow]
The time-step restriction in assumption~\ref{ass:strict-timestep} is only required for nodes where the explicit gravity load is strictly negative ($\widetilde{G}_i^{n-1} < 0$). This yields two observations:
\begin{enumerate}
    \item \textbf{Absence of gravity:} Without gravity, the advective term vanishes, and the scheme naturally respects the discrete positivity principle unconditionally for any $\tau_n > 0$.
    \item \textbf{Positivity of the gravity term:} If the local explicit gravity source term is non-negative ($G_i^{n-1} \ge 0$), positivity is guaranteed without any time-step restriction.
\end{enumerate}
However, in general flow scenarios where $G_i^{n-1} < 0$, the explicit gravity method imposes a severe time-step restriction, especially near the degenerate limit where $\theta(U_i^{n-1}) \to 0$. This motivates the linearly implicit advection scheme in the following section.
\end{remark}

\subsection{Linearly implicit advection scheme and minimum principle}\label{sec:implicit}
While the explicit gravity method preserves strict positivity under assumption~\ref{ass:strict-timestep}, this algebraic condition becomes overly restrictive near the degenerate limit. To circumvent the explicit time-step restriction when advection is active, we introduce a linearly implicit advection scheme. We revert to homogeneous Dirichlet boundary conditions ($u|_{\partial \Omega} = 0$, meaning $u_h^n \in V_h$ and $u_{b,h} = 0$).
In this section, we assume uniform time stepping, namely $\tau_n = \tau > 0$ for all $n$.

Let the advective flux be $\overline{K}(u) \mathbf{e}_z$. We define the ratio function at time $t_{n-1}$:
$$ \beta_h^{n-1}(x) := \frac{\overline{K}(u_h^{n-1}(x))}{u_h^{n-1}(x)}. $$

\begin{remark}[Degenerate limit]
As $u \to 0$, the limit $\lim_{u \to 0} \frac{\overline{K}(u)}{u}$ is finite and well-defined. This prevents division by zero at the degenerate limit.
\end{remark}

To rigorously justify that the limit $\lim_{u \to 0} \frac{\overline{K}(u)}{u}$ is finite and well-defined across different constitutive relationships, we analyze the asymptotic behavior of the auxiliary variable $u$ and the relative permeability $K_{r}(S)$ near the degenerate dry limit ($S \to 0$).

For the widely used empirical models \cite{benfanich2025}, the auxiliary variable $u(S)$ is defined such that $u \sim_{S\to 0} S$. Consequently, evaluating the limit of the advective ratio $\beta(u) = \frac{\overline{K}(u)}{u}$ is asymptotically equivalent to evaluating $K_s \frac{K_{r}(S)}{S}$.

\textbf{1. The Gardner Model \cite{gardner1958}:}
For the Gardner model, the relative permeability is linear with respect to saturation, $K_{r}(S) = S$, and the auxiliary variable behaves exactly as $u(S) = S$. Therefore, the ratio evaluates to $\frac{\overline{K}(u)}{u} = K_s$, which is a strictly bounded, finite constant.

\textbf{2. The Brooks-Corey Model \cite{brookscorey1966}:}
For the Brooks-Corey model for the capillary pressure and the relative permeability governed by a power law $K_{r}(S) = S^B$, where $B > 0$ is an empirical pore-size parameter. Because $u \sim_{S\to 0} S$ near the dry limit, the advective ratio behaves as $K_s S^{B-1}$. As established in \cite{benfanich2025}, as long as $B > 1$, the limit strictly goes to $0$.

\textbf{3. The Haverkamp Model \cite{haverkamp1977}:}
The Haverkamp model defines saturation and relative permeability via the capillary pressure $\Psi$ as $S(\Psi) = \frac{1}{1+|\alpha\Psi|^\beta}$ and $K_{r}(\Psi) = \frac{1}{1+|A\Psi|^\gamma}$. As $S \to 0$, the capillary pressure $|\Psi| \to \infty$. From the saturation equation, we can asymptotically approximate $|\Psi| \sim_{S\to 0} \frac{1}{\alpha} S^{-1/\beta}$. Substituting this into the relative permeability equation yields:
$$ K_{r}(S) \sim_{S\to 0} \frac{1}{|A\Psi|^\gamma} \sim_{S\to 0} \frac{1}{A^\gamma \alpha^{-\gamma} S^{-\gamma/\beta}} = \left(\frac{\alpha}{A}\right)^\gamma S^{\gamma/\beta}. $$
Consequently, the advective ratio near the dry limit evaluates to:
$$ \lim_{u \to 0} \frac{\overline{K}(u)}{u} = \lim_{S \to 0} K_s \left(\frac{\alpha}{A}\right)^\gamma S^{\gamma/\beta - 1}. $$
Provided that the empirical parameters satisfy $\gamma > \beta$ (which implies $\gamma/\beta - 1 > 0$), this limit goes exactly to $0$ \cite{benfanich2025}. 

\textbf{4. The van Genuchten-Mualem Model \cite{vangenuchten1980, mualem1976}:}
For the van Genuchten-Mualem model, the relative permeability is given by $K_{r}(S) = S^{1/2} [1 - (1 - S^{1/m})^m]^2$. Near the degenerate limit ($S \to 0$), we can apply a Taylor expansion to the inner term: $(1 - S^{1/m})^m \sim_{S\to 0} 1 - m S^{1/m}$. 
Substituting this back into the relative permeability yields:
$$ K_{r}(S) \sim_{S\to 0} S^{1/2} [m S^{1/m}]^2 = m^2 S^{1/2 + 2/m}. $$
Because $u \sim_{S\to 0} S$ near the dry limit, the advective ratio behaves as:
$$ \lim_{u \to 0} \frac{\overline{K}(u)}{u} = \lim_{S \to 0} \frac{K_s m^2 S^{1/2 + 2/m}}{S} = \lim_{S \to 0} K_s m^2 S^{2/m - 1/2}. $$
Since $n > 1$, we have $m = 1 - 1/n \in (0,1)$, which guarantees that $2/m > 2$. Therefore, the exponent $(2/m - 1/2)$ is strictly greater than $1.5$. Thus, the limit safely goes to $0$ \cite{benfanich2025}.

Table \ref{tab:limits} summarizes these asymptotic limits, rigorously proving that the linearly implicit advection scheme inherently avoids division by zero at the degenerate limit for all standard soil models, thereby guaranteeing the well-posedness of the local P\'eclet condition.

\begin{table}[htbp]
\centering
\caption{Asymptotic behavior and limits of the advective ratio $\beta(u) = \overline{K}(u)/u$ near the degenerate limit ($S \to 0$) for standard soil models.}
\label{tab:limits}
\resizebox{\linewidth}{!}{
\begin{tabular}{lccc}
\toprule
\textbf{Soil Model} & $K_{r}(S)$ near $S \to 0$ & $\frac{\overline{K}(u)}{u}$ & $\lim_{u \to 0} \frac{\overline{K}(u)}{u}$ \\
\midrule
\textbf{Gardner \cite{gardner1958}} & $S$ & $K_s$ & $K_s$ (Bounded) \\
\textbf{Brooks-Corey \cite{brookscorey1966}} & $S^B$ & $K_s S^{B-1}$ & $0$ (for $B > 1$) \\
\textbf{Haverkamp \cite{haverkamp1977}} & $\left(\frac{\alpha}{A}\right)^\gamma S^{\gamma/\beta}$ & $K_s \left(\frac{\alpha}{A}\right)^\gamma S^{\gamma/\beta - 1}$ & $0$ (for $\gamma > \beta$) \\
\textbf{van Genuchten-Mualem \cite{vangenuchten1980, mualem1976}} & $m^2 S^{1/2 + 2/m}$ & $K_s m^2 S^{2/m - 1/2}$ & $0$ (for $m \in (0,1)$) \\
\bottomrule
\end{tabular}
}
\end{table}

We present the linearly implicit advection scheme. For a given time step $t_n$, given $u_h^{n-1} \in V_h$, find $u_h^n \in V_h$ such that for all test functions $v_h \in V_h$:
\begin{equation} \label{eq:scheme}
\left(\frac{\theta(u_h^n)-\theta(u_h^{n-1})}{\tau}, v_h\right)_h 
+ \int_\Omega K(u_h^{n-1}) \nabla u_h^n \cdot \nabla v_h 
+ \int_\Omega \beta_h^{n-1} u_h^n \mathbf{e}_z \cdot \nabla v_h = 0.
\end{equation}

\begin{assumption}[Local Péclet spatial condition]\label{ass:local-peclet}
Let $\rho_h^{n-1}(x) := \frac{\beta_h^{n-1}(x)}{K(u_h^{n-1}(x))}$. To guarantee nodal positivity, we assume that on every simplex $T$ and for every local node $x_i$ belonging to $T$:
\begin{equation} \label{eq:local-peclet}
\sup_{x\in T} \rho_h^{n-1}(x) (\mathbf{e}_z \cdot \nabla \phi_i)^+ \le \min_{\substack{j \neq i \\ x_j \in T}} \bigl(-\nabla \phi_i \cdot \nabla \phi_j\bigr).
\end{equation}
This condition is typically verified when the local mesh size $h_T$ is sufficiently small.
\end{assumption}

\begin{remark}[Practical positivity test]
Condition (\ref{eq:local-peclet}) is a geometric mesh-Péclet restriction, not a time-step ($\Delta t$) restriction. Because $|\nabla\phi_i| \lesssim h_T^{-1}$ and $-\nabla\phi_i \cdot \nabla\phi_j \gtrsim h_T^{-2}$ on a shape-regular acute mesh, condition (\ref{eq:local-peclet}) is naturally satisfied if:
$$ Pe_T := h_T \sup_{x\in T} \rho_h^{n-1}(x) \lesssim 1. $$

For the van Genuchten-Mualem model, the ratio is
$$
\rho_h^{n-1}(x)=\alpha (n-1) \theta(u)^{1/m},
$$
where $\alpha$ and $n$ are the van Genuchten parameters, $m = 1 - 1/n$ and $n>1$. Because $\theta(u) \to 0$ as $u \to 0$, the local Péclet condition is satisfied near the degenerate limit, regardless of the mesh size. This is a critical observation that allows us to guarantee positivity near the degeneracy.
\end{remark}

\begin{theorem}[Discrete minimum principle]\label{thm:minimum}
Under assumption~\ref{ass:weakly-acute} (Weakly Acute Mesh) and assumption~\ref{ass:local-peclet} (Local Péclet Condition), if the discrete state at the previous time step is non-negative ($U^{n-1} \ge 0$), then the linearly implicit advection scheme strictly preserves this lower bound: $U^n \ge 0$.
\end{theorem}

\begin{proof}
\textbf{Step 1: Matrix form.} Let $\{\phi_i\}_{i \in I}$ be the interior nodal basis and write $u_h^n = \sum_{i \in I} U_i^n \phi_i$.
Since $\theta$ is assumed to be strictly increasing and sufficiently smooth, the mean value theorem gives, for each interior node $i \in I$, an intermediate value $\xi_i^n$ between $U_i^n$ and $U_i^{n-1}$ such that
$$ \theta(U_i^n) - \theta(U_i^{n-1}) = \theta'(\xi_i^n)\,(U_i^n - U_i^{n-1}) = d_i^n\,(U_i^n - U_i^{n-1}), $$
where $d_i^n := \theta'(\xi_i^n) \geq 0$. We define the diagonal matrix $D^n = \operatorname{diag}(d_i^n)$. The fully discrete scheme reads:
\begin{equation} \label{eq:matrix-form}
(D^n + \tau S^{n-1})U^n = D^n U^{n-1},
\end{equation}
where $D^n = \operatorname{diag}(d_i^n)$ with $d_i^n \ge 0$, and the system matrix is $S^{n-1} = A^{n-1} + C^{n-1}$. The linearly implicit (lagged-coefficient) advective matrix is $C_{ij}^{n-1} = \int_\Omega \beta_h^{n-1} \phi_j \mathbf{e}_z \cdot \nabla \phi_i \, dx$.

\textbf{Step 2: Non-positive off-diagonal entries.} Fix a simplex $T$. For any two distinct local nodes $i \neq j$, the local matrix entry is:
$$ S_{ij}^T = \int_T K(u_h^{n-1}) \nabla \phi_j \cdot \nabla \phi_i + \int_T \beta_h^{n-1} \phi_j \mathbf{e}_z \cdot \nabla \phi_i. $$
Defining $\mu_{ij}^T := -\nabla \phi_i \cdot \nabla \phi_j \ge 0$, and observing that $0 \le \phi_j \le 1$, we bound the local entry:
$$ S_{ij}^T \le \Bigl[ -\mu_{ij}^T + \Bigl(\sup_{x\in T} \rho_h^{n-1}(x)\Bigr) (\mathbf{e}_z \cdot \nabla \phi_i)^+ \Bigr] \int_T K(u_h^{n-1}). $$
Under assumption~\ref{ass:local-peclet}, the bracketed term is less than or equal to zero. Hence, after global assembly, all off-diagonal entries are non-positive.

\textbf{Step 3: Column sums.} Let $S_{\mathrm{full}}^{n-1}$ be the extended matrix evaluated on all nodes $I \cup \Gamma$. Using the partition of unity ($\sum_{i \in I \cup \Gamma} \phi_i \equiv 1 \implies \sum_{i \in I \cup \Gamma} \nabla \phi_i = \mathbf{0}$), the column sum for any interior node $j \in I$ is computed as follows:
$$
\begin{aligned}
\sum_{i \in I \cup \Gamma} (S_{\mathrm{full}}^{n-1})_{ij}
&= \int_\Omega K(u_h^{n-1}) \nabla \phi_j \cdot \left( \sum_{i \in I \cup \Gamma} \nabla \phi_i \right) dx \\
&\quad + \int_\Omega \beta_h^{n-1} \phi_j \mathbf{e}_z \cdot \left( \sum_{i \in I \cup \Gamma} \nabla \phi_i \right) dx \\
&= 0.
\end{aligned}
$$
Restricting this evaluation to the interior block $S^{n-1}$ for $i, j \in I$ gives:
$$ \sum_{i \in I} S_{ij}^{n-1} = \sum_{i \in I \cup \Gamma} (S_{\mathrm{full}}^{n-1})_{ij} - \sum_{i \in \Gamma} (S_{\mathrm{full}}^{n-1})_{ij} = - \sum_{i \in \Gamma} (S_{\mathrm{full}}^{n-1})_{ij}. $$
Since $i \in \Gamma$ and $j \in I$ implies $i \neq j$, Step 2 guarantees $(S_{\mathrm{full}}^{n-1})_{ij} \le 0$. Thus:
$$ \sum_{i \in I} S_{ij}^{n-1} \ge 0. $$

\textbf{Step 4: M-matrix inverse.} The matrix $B^n = D^n + \tau S^{n-1}$ has non-positive off-diagonal entries and non-negative column sums. Because the interior mesh forms a connected graph terminating at the Dirichlet boundary, $B^n$ is an irreducibly diagonally dominant M-matrix \cite{varga2000}. By the fundamental characterizations of non-singular M-matrices \cite{berman1994}, its inverse is strictly non-negative: $(B^n)^{-1} \ge 0$. Multiplying the non-negative right-hand side $D^n U^{n-1}$ by $(B^n)^{-1}$ guarantees $U^n \ge 0$.
\end{proof}

\subsection{Discrete maximum principle}\label{sec:dmp}
While Section~\ref{sec:implicit} establishes the minimum principle (lower bound) of the solution, this section proves the complementary discrete maximum principle, establishing the global upper bound. Let $\{\phi_j\}_{j\in I\cup\Gamma}$ denote the full nodal basis spanning interior ($I$) and boundary ($\Gamma$) nodes. The full assembled matrix is:
$$ (S_{\mathrm{full}}^{n-1})_{ij} := \int_\Omega K(u_h^{n-1}) \nabla \phi_j \cdot \nabla \phi_i \, dx + \int_\Omega \beta_h^{n-1}\phi_j\,\mathbf g\cdot \nabla \phi_i \, dx. $$
The interior nodal equations can be written as:
\begin{equation}\label{eq:nodal-dmp-full}
m_i \frac{\theta(U_i^n)-\theta(U_i^{n-1})}{\tau} + \sum_{j\in I\cup\Gamma}(S_{\mathrm{full}}^{n-1})_{ij} U_j^n = 0, \qquad i\in I,
\end{equation}
with homogeneous Dirichlet boundary values $U_j^n = 0, \ j\in \Gamma$.

\begin{assumption}[Discrete row-sum condition \cite{barrenechea2024}, Def. 3.2]\label{ass:dmp-rowsum}
Assume that for every interior node $i\in I$, we have:
\begin{equation}\label{eq:dmp-rowsum}
\sum_{j\in I\cup\Gamma}(S_{\mathrm{full}}^{n-1})_{ij} \ge 0.
\end{equation}
Since the basis functions satisfy $\sum_{j\in I\cup\Gamma}\phi_j \equiv 1$ the equation above is equivalent to:
$$ \sum_{j\in I\cup\Gamma}(S_{\mathrm{full}}^{n-1})_{ij} = \int_\Omega \beta_h^{n-1}\,\mathbf g\cdot \nabla \phi_i \, dx \ge 0. $$
\end{assumption}

\begin{theorem}[Discrete maximum principle]\label{thm:maximum}
Assume that $\theta:\mathbb R\to\mathbb R$ is strictly increasing, that $K(u)\ge 0$ and $\overline K(u)\ge 0$, that $U_j^{n-1} \ge 0$ for all $j\in I$, and that assumptions~\ref{ass:weakly-acute}, \ref{ass:local-peclet}, and \ref{ass:dmp-rowsum} are satisfied. Then the fully discrete scheme satisfies:
$$ U_i^n \le \max_{j\in I} U_j^{n-1}, \qquad \forall i\in I. $$
\end{theorem}

\begin{proof}
As established in Step 2 of Theorem~\ref{thm:minimum}, the weakly acute mesh Condition and the Local Péclet Condition guarantee that all off-diagonal entries are non-positive: $(S_{\mathrm{full}}^{n-1})_{ij} \le 0$ for all $i\neq j$.

Assume by contradiction that there exists an interior node $i_*\in I$ such that:
$$ U_{i_*}^n = \max_{j\in I\cup\Gamma} U_j^n > \max_{j\in I} U_j^{n-1}. $$
This implies $U_{i_*}^n>0$ and $U_{i_*}^n>U_{i_*}^{n-1}$. Starting from the interior nodal equation \eqref{eq:nodal-dmp-full} and adding and subtracting $U_{i_*}^n\sum_{j\neq i_*}(S_{\mathrm{full}}^{n-1})_{i_*j}$, we obtain:
$$ m_{i_*}\frac{\theta(U_{i_*}^n)-\theta(U_{i_*}^{n-1})}{\tau} + \Biggl(\sum_{j\in I\cup\Gamma}(S_{\mathrm{full}}^{n-1})_{i_*j}\Biggr) U_{i_*}^n = \sum_{j\neq i_*}(S_{\mathrm{full}}^{n-1})_{i_*j}\bigl(U_{i_*}^n-U_j^n\bigr). $$
Since $U_{i_*}^n\ge U_j^n$ and the off-diagonal entries are non-positive, the right-hand side is non-positive. By assumption~\ref{ass:dmp-rowsum} (the row-sum condition) and $U_{i_*}^n>0$, the second term on the left-hand side is $\ge 0$. Therefore, we must have:
$$ m_{i_*}\frac{\theta(U_{i_*}^n)-\theta(U_{i_*}^{n-1})}{\tau} \le 0. $$
Because $m_{i_*}>0$, $\tau>0$, and $\theta$ is strictly increasing, this implies $U_{i_*}^n \le U_{i_*}^{n-1}$, which explicitly contradicts the initial assumption. Hence the upper bound holds.
\end{proof}

\section{Numerical tests}\label{sec:numerical}
This section reports numerical tests for the explicit gravity method and linearly implicit advection schemes. All numerical tests in this section solve the Richards model introduced in Section 1. At each time step, the nonlinear discrete system is solved with Newton's method. The van Genuchten--Mualem constitutive relations are written in terms of $\theta(u)$:
\begin{align}
K_{\text{rel}}(\theta(u)) &= K_s \sqrt{\theta(u)}\left(1 - \left(1 - \theta(u)^{1/m}\right)^m\right)^2, \\
K(u) &= \frac{h_{\text{cap}}}{n-1} \, K_{\text{rel}}(\theta(u)) \, \theta(u)^{-1/m}, \\
\overline{K}(u) &= K_{\text{rel}}(\theta(u)).
\end{align}
Here, $h_{\text{cap}}=1/\alpha$, $m=1-1/n$, $\theta(u) \in [0,1]$ is the effective saturation, $\theta_s$ is the saturated water content, $\theta_r$ is the residual water content, $\phi = \theta_s - \theta_r$ is the effective porosity, and the finite element space is conforming $\mathbb P_1$. All numerical simulations are implemented and executed using the open-source finite element software FreeFEM++ \cite{hecht2012}. At each Newton iteration, the resulting sparse linear systems are solved using the UMFPACK direct solver. The Newton method is configured with a tolerance of $10^{-6}$ and a maximum of $100$ iterations per time step to ensure robust nonlinear convergence, but on average it converges in $1 - 5$ iterations.


\subsection{Test 1: condition not verified while positivity is maintained}
\label{sec:test1}
We first report a single-domain test using the explicit gravity method on $\Omega=(0,100)\times(0,100)$. The homogeneous soil parameters are as follows:
\begin{center}
\begin{tabular}{cccccc}
\toprule
$K_s$ & $\alpha$ & $n$ & $\theta_s$ & $\theta_r$ & $\phi = \theta_s - \theta_r$ \\
\midrule
$1.0$ & $0.02$ & $2.0$ & $0.48$ & $0.08$ & $0.40$ \\
\bottomrule
\end{tabular}
\end{center}
The simulation is run with time step $\tau=1$ and final time $T=5$.
The spatial discretization is conforming $\mathbb P_1$ on a mesh with 20 segments on each boundary side.
The imposed initial and boundary conditions are
\[
\theta(u_0)(x,z)=
\begin{cases}
1, & 0<z<50,\\
0, & 50\le z<100,
\end{cases}
\quad (x,z)\in\Omega,
\]
with $\theta(u)(x,0,t)=1$, $\theta(u)(x,100,t)=0$ for $0<x<100$, and
$K(u)\,\nabla u\cdot n=0$ on $x=0$ and $x=100$.

At each step, we evaluate
$$ \tau_{\mathrm{crit}}^{n-1}=\min_{i:\,G_i^{n-1}<0}\frac{m_i\,\theta(U_i^{n-1})}{|G_i^{n-1}|}, \qquad
\mu_{\min}^{n-1}=\min_i\bigl(m_i\,\theta(U_i^{n-1})+\tau G_i^{n-1}\bigr), $$
and we record $\theta(U)_{\min}$ and $\theta(U)_{\max}$ after convergence.

Table~\ref{tab:first-test-explicit} shows that for all reported steps the explicit condition is not verified
($\mu_{\min}^{n-1}<0$), while the computed solution remains nonnegative
($\theta(U)_{\min}=0$).

\begin{table}[htbp]
\centering
\caption{Test~\ref{sec:test1} using the explicit gravity method: positivity diagnostics and saturation bounds ($\tau=1$, $T=5$).}
\label{tab:first-test-explicit}
\begin{tabular}{cccccc}
\toprule
$t^n$ & $\tau_{\mathrm{crit}}^{n-1}$ & $\mu_{\min}^{n-1}$ & Condition met? & $\theta(U)_{\min}$ & $\theta(U)_{\max}$ \\
\midrule
$1$ & $0$ & $-8.10899\times10^{-1}$ & No & $0$ & $1$ \\
$2$ & $0$ & $-1.00297\times10^{-2}$ & No & $0$ & $1$ \\
$3$ & $0$ & $-5.68438\times10^{-6}$ & No & $0$ & $1$ \\
$4$ & $0$ & $-2.05524\times10^{-22}$ & No & $0$ & $1$ \\
$5$ & $0$ & $-7.45114\times10^{-31}$ & No & $0$ & $1$ \\
\bottomrule
\end{tabular}
\end{table}

In this simulation, the condition is not verified on $5$ out of $5$ steps, and the computed saturation remains nonnegative on those $5$ steps.

\subsection{Test 2: 2D column positivity-violation run using the explicit gravity method}
\label{sec:test2}
This test uses the common constitutive setup above. The test-specific soil parameters, spatial discretization, initial/boundary data, and time step are listed below.
We consider a vertical rectangular column $\Omega = (0, L) \times (0, H)$ with horizontal coordinate $x$ and vertical coordinate $z$, where $L = 50$ and $H = 200$. The domain is discretized with a conforming $\mathbb{P}_1$ triangulation using 20 segments on each horizontal side and 40 segments on each vertical side. The soil parameters are as follows:
\begin{center}
\begin{tabular}{cccccc}
\toprule
$K_s$ & $\alpha$ & $n$ & $\theta_s$ & $\theta_r$ & $\phi = \theta_s - \theta_r$ \\
\midrule
$5.0$ & $0.05$ & $2.0$ & $0.45$ & $0.05$ & $0.40$ \\
\bottomrule
\end{tabular}
\end{center}

The initial condition is a sharp front in effective saturation:
$$ \theta(u_0)(x,z) = \begin{cases} 1.0 & \text{if } z < 0.3H, \\ 0.2 & \text{if } z \ge 0.3H, \end{cases} $$
which corresponds to a fully saturated lower region and an unsaturated upper region. The imposed boundary conditions are
$$
\theta(u)(x,H,t)=0.2,\qquad \theta(u)(x,0,t)=1.0,\qquad 0<x<L,
$$
$$
K(u)\,\nabla u\cdot n=0\qquad \text{on } x=0 \text{ and } x=L.
$$
The strictly positive boundary values are consistent with assumption~\ref{ass:init-boundary} and ensure that the time-step restriction remains well-defined.

The time step is set to $\tau = 5.0$ with final time $T = 50$. At the initial sharp front, the local critical time step is $\tau_{\text{crit}} \approx 0.31$.

\subsubsection{Computed diagnostics}
At each time step $t_n$, we evaluate the following quantities from the discrete solution $U^{n-1}$ to assess compliance with assumption~\ref{ass:strict-timestep}.

\textbf{Explicit advective load.} For each interior node $i \in I$, we compute the gravity-driven load defined in Section~\ref{sec:explicit}:
$G_i^{n-1} = -\int_\Omega \overline{K}(u_h^{n-1})\, \mathbf{e}_z \cdot \nabla \phi_i \, dx$.
This quantity measures the net explicit advective flux at node $i$, assembled from the finite element integrals using the known solution $u_h^{n-1}$.

\textbf{Nodal margin.} Using $G_i^{n-1}$ and the lumped mass $m_i = \int_\Omega \phi_i \, dx$, we form the pointwise margin
$\mu_i^{n-1} := m_i\, \theta(U_i^{n-1}) + \tau\, G_i^{n-1}$.
We track the global minimum $\mu_{\min}^{n-1} = \min_{i \in I} \mu_i^{n-1}$.

\textbf{Critical time step.} For each node where $G_i^{n-1} < 0$, the maximal allowable time step at that node is
$\tau_{\text{crit},i}^{n-1} = \frac{m_i\, \theta(U_i^{n-1})}{|G_i^{n-1}|}$.
The global critical time step is $\tau_{\text{crit}}^{n-1} = \min_{i:\, G_i^{n-1} < 0} \tau_{\text{crit},i}^{n-1}$. The condition $\tau < \tau_{\text{crit}}^{n-1}$ is equivalent to $\mu_{\min}^{n-1} > 0$.

\begin{remark}
In the above, we use $G_i^{n-1}$ rather than the modified load from Section~\ref{sec:explicit}:
$$
\widetilde{G}_i^{n-1} = G_i^{n-1} - \sum_{j \in \Gamma} A_{ij}^{n-1} u_{b,h}(x_j).
$$
The boundary correction $-A_{ij}^{n-1} u_{b,h}(x_j)$ is non-negative.
\end{remark}

\subsubsection{Results}

Table~\ref{tab:explicit-condition} reports $\tau_{\text{crit}}^{n-1}$, $\mu_{\min}^{n-1}$, and the number of interior nodes with $\mu_i^{n-1} < 0$ at each time step. At $t = 5$, the critical time step is $\tau_{\text{crit}} = 0.312$ and $\mu_{\min} = -121.01$. After the first step, $\tau_{\text{crit}}$ is zero in the table because $\theta(U_i^{n-1}) < 0$ at some nodes.

\begin{table}[htbp]
\centering
\caption{Explicit positivity diagnostics at each time step.}
\label{tab:explicit-condition}
\begin{tabular}{ccccc}
\toprule
$t^n$ & $\tau_{\text{crit}}^{n-1}$ & $\mu_{\min}^{n-1}$ & Nodes with $\mu_i^{n-1}<0$ & Condition is satisfied? \\
\midrule
$5$  & $0.311894$ & $-121.006$ & $23$  & No \\
$10$ & $0$        & $-12.1108$ & $122$ & No \\
$15$ & $0$        & $-19.3902$ & $99$  & No \\
$20$ & $0$        & $-19.7793$ & $97$  & No \\
$25$ & $0$        & $-19.6239$ & $98$  & No \\
$30$ & $0$        & $-19.5821$ & $96$  & No \\
$35$ & $0$        & $-19.5722$ & $96$  & No \\
$40$ & $0$        & $-19.5660$ & $95$  & No \\
$45$ & $0$        & $-19.5608$ & $96$  & No \\
$50$ & $0$        & $-19.5564$ & $98$  & No \\
\bottomrule
\end{tabular}
\end{table}

Table~\ref{tab:saturation-bounds} reports $\theta(U)_{\min} = \min_{i \in I} \theta(U_i^n)$ and $\theta(U)_{\max} = \max_{i \in I} \theta(U_i^n)$ after convergence at each time step. The computed range is $\theta(U)_{\min} \in [-0.374, -0.370]$ and $\theta(U)_{\max} = 1.000$ for all reported steps. Hence, in this test using the explicit gravity method the lower limit $0$ is not satisfied, while the upper limit $1$ is attained.

\begin{table}[htbp]
\centering
\caption{Minimal and maximal values of $\theta(U)$ after convergence.}
\label{tab:saturation-bounds}
\begin{tabular}{ccc}
\toprule
$t^n$ & Min $\theta(U)$ & Max $\theta(U)$ \\
\midrule
$5$  & $\mathbf{-0.374172}$ & $1.000$ \\
$10$ & $\mathbf{-0.374110}$ & $1.000$ \\
$15$ & $\mathbf{-0.373987}$ & $1.000$ \\
$20$ & $\mathbf{-0.373781}$ & $1.000$ \\
$25$ & $\mathbf{-0.373478}$ & $1.000$ \\
$30$ & $\mathbf{-0.373069}$ & $1.000$ \\
$35$ & $\mathbf{-0.372552}$ & $1.000$ \\
$40$ & $\mathbf{-0.371929}$ & $1.000$ \\
$45$ & $\mathbf{-0.371205}$ & $1.000$ \\
$50$ & $\mathbf{-0.370387}$ & $1.000$ \\
\bottomrule
\end{tabular}
\end{table}

To verify that the discrete minimum principle holds when the explicit time step restriction is satisfied, we re-run this exact configuration but with a reduced time step $\tau = 0.25$. As shown in Table~\ref{tab:explicit-verification}, this choice satisfies $\tau \le \tau_{\mathrm{crit}}^{n-1}$ at all reported times, maintaining a strictly positive margin $\mu_{\min}^{n-1}$. Consequently, the computed saturation $\theta(U)_{\min}$ remains bounded by the theoretical minimum value of $0.2$, confirming that the negative values observed in the simulation using $\tau = 5.0$ are purely an artifact of violating the conditional stability bound on the time step.

\begin{table}[htbp]
\centering
\caption{Explicit gravity method verification with stable time step $\tau = 0.25$.}
\label{tab:explicit-verification}
\begin{tabular}{cccccc}
\toprule
$t^n$ & $\tau_{\mathrm{crit}}^{n-1}$ & $\mu_{\min}^{n-1}$ & Condition met? & $\theta(U)_{\min}$ & $\theta(U)_{\max}$ \\
\midrule
$0.25$ & $0.311894$ & $0.127551$ & Yes & $0.200$ & $1.000$ \\
$0.50$ & $2.47070$ & $0.127551$ & Yes & $0.200$ & $1.000$ \\
$0.75$ & $2.87201$ & $0.127551$ & Yes & $0.200$ & $1.000$ \\
$1.00$ & $1.94169$ & $0.127551$ & Yes & $0.200$ & $1.000$ \\
$1.25$ & $1.59382$ & $0.127551$ & Yes & $0.200$ & $1.000$ \\
$1.50$ & $1.42337$ & $0.127551$ & Yes & $0.200$ & $1.000$ \\
$1.75$ & $1.32307$ & $0.127551$ & Yes & $0.200$ & $1.000$ \\
$2.00$ & $1.25672$ & $0.127551$ & Yes & $0.200$ & $1.000$ \\
$2.25$ & $1.20937$ & $0.127551$ & Yes & $0.200$ & $1.000$ \\
$2.50$ & $1.17378$ & $0.127551$ & Yes & $0.200$ & $1.000$ \\
\bottomrule
\end{tabular}
\end{table}

\subsection{Comparison test: linearly implicit advection scheme}
This comparison reuses exactly the setup of Test~\ref{sec:test2}: same soil parameters, domain, spatial discretization, initial/boundary data, and time stepping ($\tau=5.0$, $T=50$).
The imposed initial and boundary conditions are therefore
$$
\theta(u_0)(x,z) =
\begin{cases}
1.0, & z < 0.3H,\\
0.2, & z \ge 0.3H,
\end{cases}
$$
$$
\theta(u)(x,H,t)=0.2,\qquad \theta(u)(x,0,t)=1.0,\qquad K(u)\,\nabla u\cdot n=0\ \text{on } x=0,L.
$$
We re-run the same test using the linearly implicit advection scheme~\eqref{eq:scheme} from Section~\ref{sec:implicit}. This scheme replaces the explicit load $\int_\Omega \overline{K}(u_h^{n-1})\, \mathbf{e}_z \cdot \nabla v_h$ by the bilinear form $\int_\Omega \beta_h^{n-1}\, u_h^n\, \mathbf{e}_z \cdot \nabla v_h$, where $\beta_h^{n-1} = \overline{K}(u_h^{n-1}) / u_h^{n-1}$.

Table~\ref{tab:comparison} reports the minimal saturation $\theta(U)_{\min}$ after convergence for both schemes at each time step. The explicit scheme gives $\theta(U)_{\min} \approx -0.374$. The linearly implicit advection scheme gives $\theta(U)_{\min} = 0.2$. In the same simulation, both schemes have $\theta(U)_{\max}=1.0$ at all reported times. Both the minimal and maximal saturation values are therefore monitored at each reported step in this comparison. The explicit scheme violates the lower limit $0$, whereas the linearly implicit scheme stays within $[0.2,1.0]$. Figure~\ref{fig:positivity-profile} plots the vertical saturation profile along the centerline at the final time.

\begin{table}[htbp]
\centering
\caption{Comparison of minimal saturation $\theta(U)_{\min}$ between explicit and linearly implicit advection schemes.}
\label{tab:comparison}
\begin{tabular}{ccc}
\toprule
$t^n$ & Explicit $\theta(U)_{\min}$ & Linearly implicit $\theta(U)_{\min}$ \\
\midrule
$5$  & $\mathbf{-0.374172}$ & $0.200$ \\
$10$ & $\mathbf{-0.374110}$ & $0.200$ \\
$15$ & $\mathbf{-0.373987}$ & $0.200$ \\
$20$ & $\mathbf{-0.373781}$ & $0.200$ \\
$25$ & $\mathbf{-0.373478}$ & $0.200$ \\
$30$ & $\mathbf{-0.373069}$ & $0.200$ \\
$35$ & $\mathbf{-0.372552}$ & $0.200$ \\
$40$ & $\mathbf{-0.371929}$ & $0.200$ \\
$45$ & $\mathbf{-0.371205}$ & $0.200$ \\
$50$ & $\mathbf{-0.370387}$ & $0.200$ \\
\bottomrule
\end{tabular}
\end{table}

\begin{figure}[htbp]
    \centering
    \includegraphics[width=0.65\textwidth]{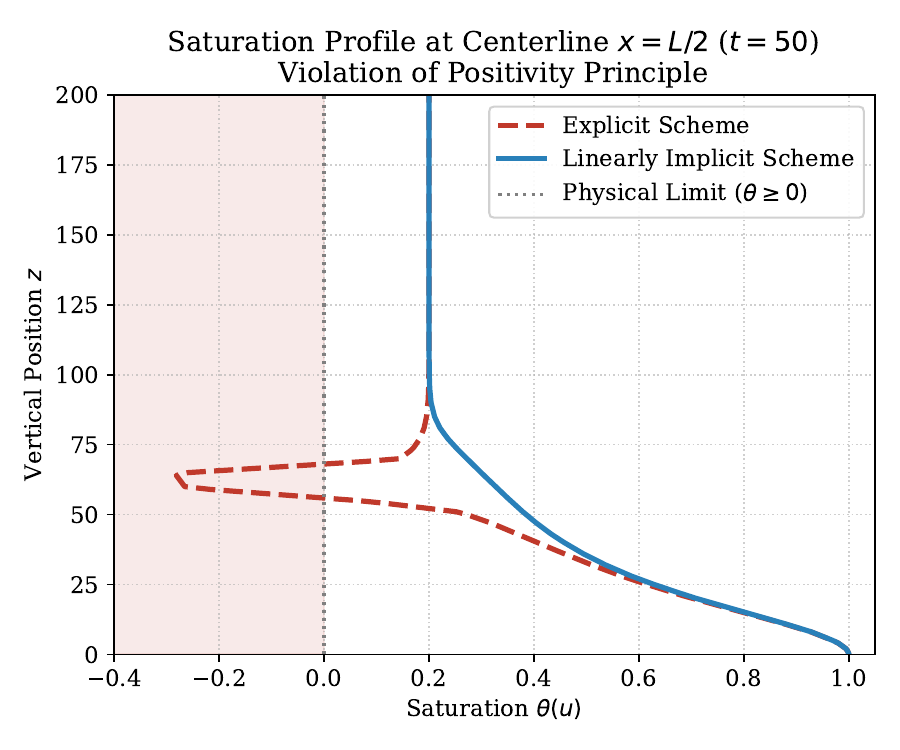}
    \caption{Vertical saturation profile along the centerline $x=L/2$ at the final time $t=50$.}
    \label{fig:positivity-profile}
\end{figure}

\subsection{Test 3: both assumptions verified — discrete maximum principle holds}
\label{sec:test-verified}
Theorem~\ref{thm:maximum} guarantees the discrete maximum principle when both assumption~\ref{ass:local-peclet} (cell Péclet condition) and assumption~\ref{ass:dmp-rowsum} (non-negative row-sum condition) are satisfied. This test constructs a scenario in which both assumptions hold simultaneously, thereby verifying the theorem in its positive direction.

The test uses a structured $\mathbb{P}_1$ mesh (all angles $\le 90^\circ$, guaranteeing the nonnegative-type property of the stiffness matrix), with the following parameters:
\begin{center}
\begin{tabular}{cccccc}
\toprule
$K_s$ & $\alpha$ & $n$ & $\theta_s$ & $\theta_r$ & $\phi = \theta_s - \theta_r$ \\
\midrule
$1.0$ & $0.01$ & $2.0$ & $0.45$ & $0.05$ & $0.40$ \\
\bottomrule
\end{tabular}
\end{center}
The small value $\alpha = 0.01$ yields a large capillary suction scale $h_{\text{cap}} = 1/\alpha = 100$, making diffusion strongly dominant and ensuring the maximal cell Péclet number $Pe_c \ll 1$ across the entire domain and time interval.

The domain is $\Omega = [0,50]\times[0,100]$ with $\tau = 1.0$ and $T = 20$. A monotone Dirichlet profile is imposed: the bottom boundary is kept wet ($\theta = 0.8$), the top boundary dry ($\theta = 0.2$), and the sides follow the corresponding linear profile. The initial condition is the same linear profile. Under this monotone configuration, the advective coefficient $\beta_h^{n-1}$ decreases in the upward direction, so the discrete row-sum
$$
\sum_{j \in I \cup \Gamma} S_{\text{full},ij} = \int_\Omega \beta_h^{n-1} \mathbf{e}_z \cdot \nabla \phi_i \, dx > 0
$$
is strictly positive for every interior node $i \in I$ at every time step.

Figure~\ref{fig:verified-dmp} reports the three diagnostics over $t \in [1, 20]$. Panel~(a) confirms that the minimal row-sum over interior nodes remains strictly positive throughout the simulation. Panel~(b) confirms that the maximal cell Péclet number stays well below~$1$. Panel~(c) shows that the computed saturation remains exactly within $[0.2, 0.8]$ at all times, verifying that the discrete maximum principle is preserved as predicted by Theorem~\ref{thm:maximum}.

\begin{figure}[htbp]
    \centering
    \includegraphics[width=\textwidth]{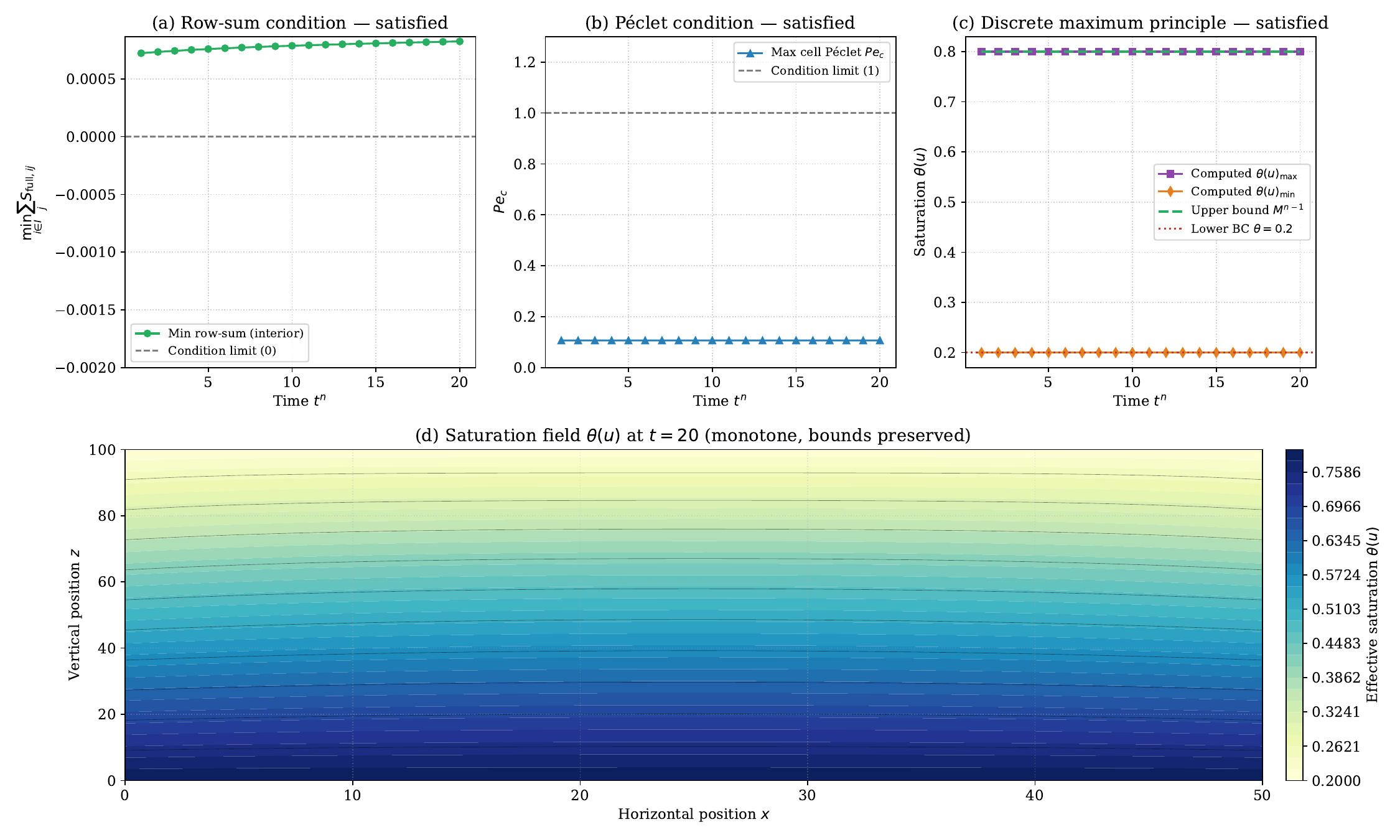}
    \caption{Test~\ref{sec:test-verified}: Both assumptions satisfied. (a) Minimal row-sum over interior nodes (strictly positive throughout). (b) Maximal cell Péclet number (well below~$1$ throughout). (c) Computed saturation bounds $[\theta_{\min},\theta_{\max}]$ remain within the Dirichlet values $[0.2, 0.8]$, confirming the discrete maximum principle. (d) Saturation field $\theta(u)$ at $t=20$: smooth monotone profile from wet bottom ($\theta=0.8$) to dry top ($\theta=0.2$), with all values within the prescribed bounds.}
    \label{fig:verified-dmp}
\end{figure}

\subsection{Test 4: discrete maximum principle and row-sum condition violation}
\label{sec:test-violation}
Theorem~\ref{thm:maximum} shows that the linearly implicit advection scheme satisfies the discrete maximum principle, under assumption~\ref{ass:dmp-rowsum}. This assumption requires non-negative row sums in the system matrix, so $B^n$ is an M-matrix. Geometrically, this restricts the discrete divergence of the advective velocity field $\beta_h^{n-1} \mathbf{e}_z$.

This test uses the same soil parameter set as Test~\ref{sec:test2}, but a different time step: $\tau = 1.0$. The parameters are as follows:

\begin{center}
\begin{tabular}{cccccc}
\toprule
$K_s$ & $\alpha$ & $n$ & $\theta_s$ & $\theta_r$ & $\phi = \theta_s - \theta_r$ \\
\midrule
$5.0$ & $0.05$ & $2.0$ & $0.45$ & $0.05$ & $0.40$ \\
\bottomrule
\end{tabular}
\end{center}
The spatial discretization is conforming $\mathbb P_1$ on a mesh with 20 horizontal and 40 vertical boundary segments, and the reported run uses final time $T=20$.
The imposed initial and boundary conditions are
$$
\theta(u_0)(x,z)=
\begin{cases}
1.0, & 0.55H \le z \le 0.75H,\\
0.2, & \text{otherwise},
\end{cases}
$$
$$
\theta(u)(x,z,t)=0.2\qquad \text{for } (x,z)\in\partial\Omega.
$$
Under these conditions, the row-sum
$$ \sum_{j \in I \cup \Gamma} S_{\text{full},ij} = \int_\Omega \beta_h^{n-1} \, \mathbf{e}_z \cdot \nabla \phi_i \, dx $$
becomes negative for nodes near the front, violating assumption~\ref{ass:dmp-rowsum}.

Figure~\ref{fig:dmp-contour} shows the spatial distribution of the saturation patch at the final time $t=20$. 

\begin{figure}[htbp]
    \centering
    \includegraphics[width=\textwidth]{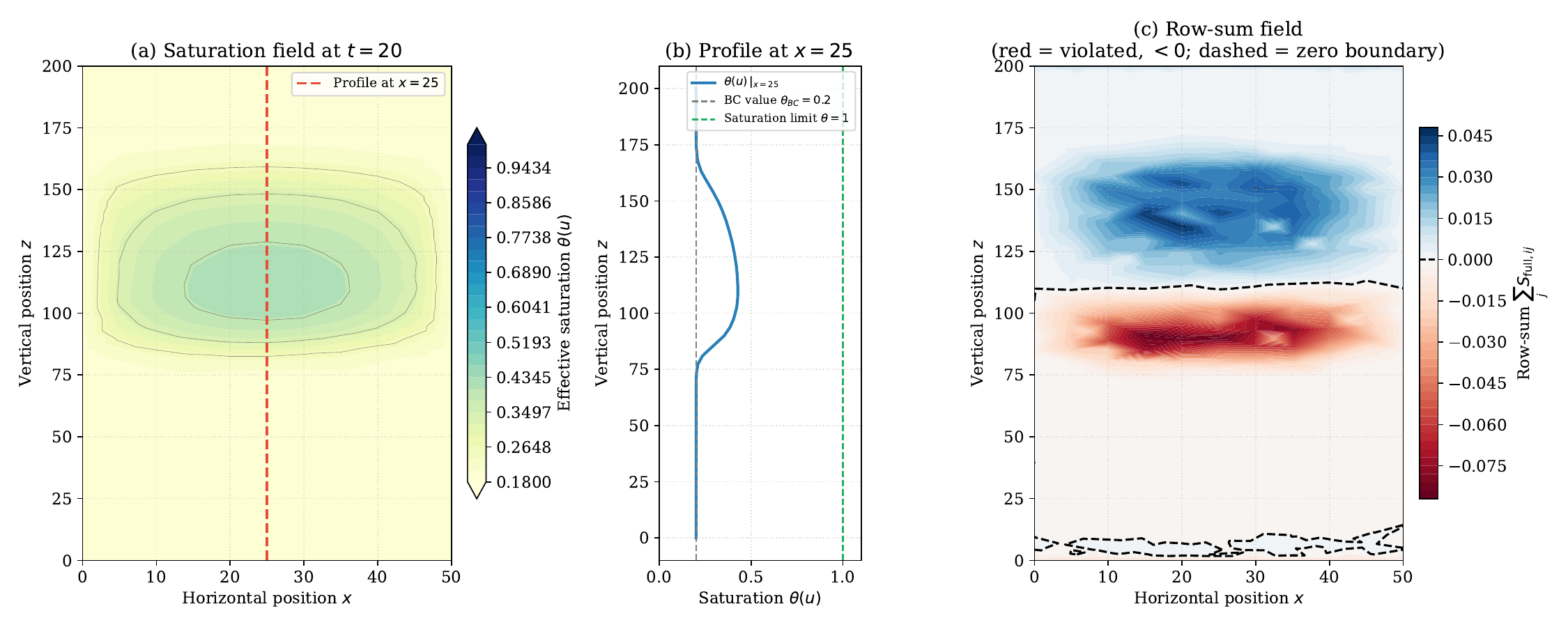}
    \caption{Contour plot of spatial saturation $\theta(u)$ at $t=20$.}
    \label{fig:dmp-contour}
\end{figure}

Figure~\ref{fig:dmp-test} presents the matrix diagnostics. Panel (a) plots the minimal row-sum across all interior nodes at each time step. The row-sum is negative for the reported steps.

\begin{figure}[htbp]
    \centering
    \includegraphics[width=\textwidth]{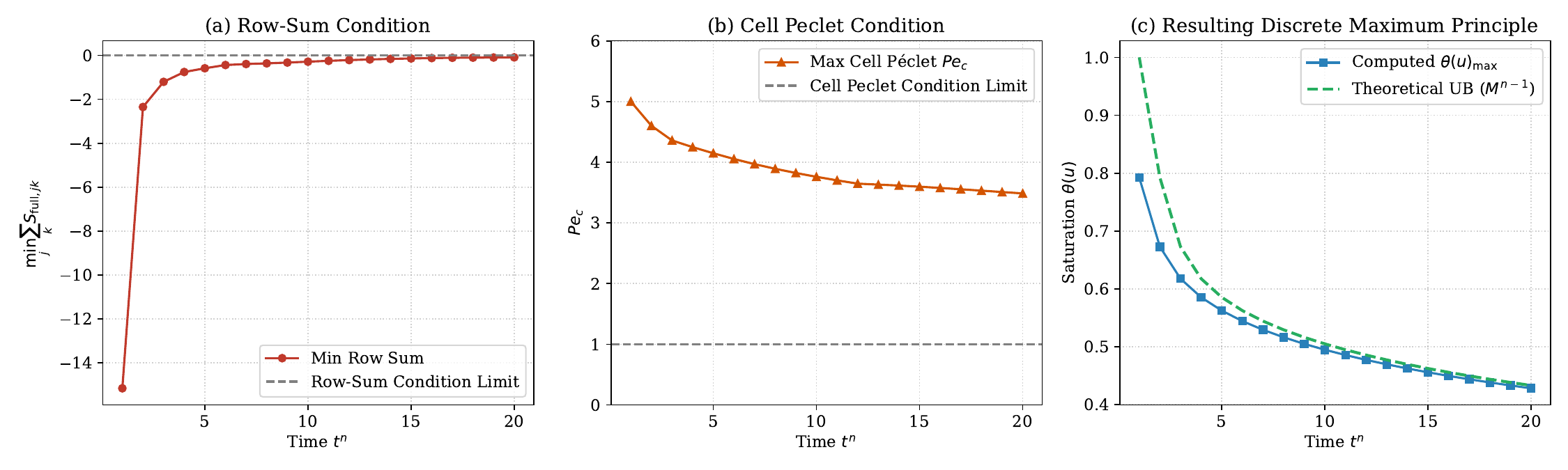}
    \caption{Time evolution of the linearly implicit advection scheme. (a) Minimum row-sum across the domain. (b) Maximal cell Péclet number $Pe_T$ (cell Péclet condition). (c) Evolution of the maximal computed saturation and of the upper bound $M^{n-1}$.}
    \label{fig:dmp-test}
\end{figure}

Panel (c) shows that the maximal computed saturation $\theta(u)_{\max}^n$ decreases over time and stays below the dynamic upper bound
$$ M^{n-1} := \max\big(\max_{j \in I} \theta(U_j^{n-1}), \max_{j \in \Gamma} \theta(U_{BC, j}^n)\big), $$
where $U_{BC,j}^n$ denotes the prescribed Dirichlet value at boundary node $j$.

From the numerical history in this test, the extrema over $t \in [1,20]$ are $\theta(u)_{\min}=0.200$ and $\theta(u)_{\max}=0.792383$. Therefore, both limits $0$ and $1$ are satisfied in this case.

\subsection{Test 5: advection-dominated regime}
\label{sec:test-advection}
This subsection reports a one-dimensional advection-dominated test on the interval $\Omega=(0,H)$ with $H=200$.

This test uses a different setup from Test~\ref{sec:test-verified} for domain, mesh, and final time.
It uses conforming $\mathbb P_1$ discretization with $N=40$ mesh points ($N-1=39$ segments), time step $\tau=1.0$, and final time $T=10$.
The parameters $n=2$, $\theta_s=0.45$, and $\theta_r=0.05$ are reused from Test~\ref{sec:test2} and Test~\ref{sec:test-verified}, while $K_s$ and $\alpha$ are modified to construct a more extreme advection-dominated case. Specifically, the parameter $\alpha$ is taken to be $1.0$, which is $20$ times larger than the value used in the previous tests ($\alpha=0.05$). This modification significantly decreases the capillary diffusion scale, severely driving the system into an advection-dominated regime:
\begin{center}
\begin{tabular}{cccccc}
\toprule
$K_s$ & $\alpha$ & $n$ & $\theta_s$ & $\theta_r$ & $\phi = \theta_s - \theta_r$ \\
\midrule
$10.0$ & $1.0$ & $2.0$ & $0.45$ & $0.05$ & $0.40$ \\
\bottomrule
\end{tabular}
\end{center}
With this choice, the capillary rise is $h_{\text{cap}}=1/\alpha=1.0$.
The imposed initial and boundary conditions are
$$
\theta(u_0)(z)=
\begin{cases}
1.0, & z>H/2,\\
0.2, & z\le H/2,
\end{cases}
$$
$$
\theta(u)(0,t)=0.2,\qquad \theta(u)(H,t)=1.0.
$$

Figure~\ref{fig:overshoot-plot} shows the results of this test. Panel (a) reports negative row-sum values (minimal value $-5.563$ over $t\in[1,10]$). Panel (b) reports a maximal Péclet indicator of $16.743$. Panel (c) reports the computed maximum of $\theta$ reaching values above $6$.

\begin{figure}[htbp]
    \centering
    \includegraphics[width=\textwidth]{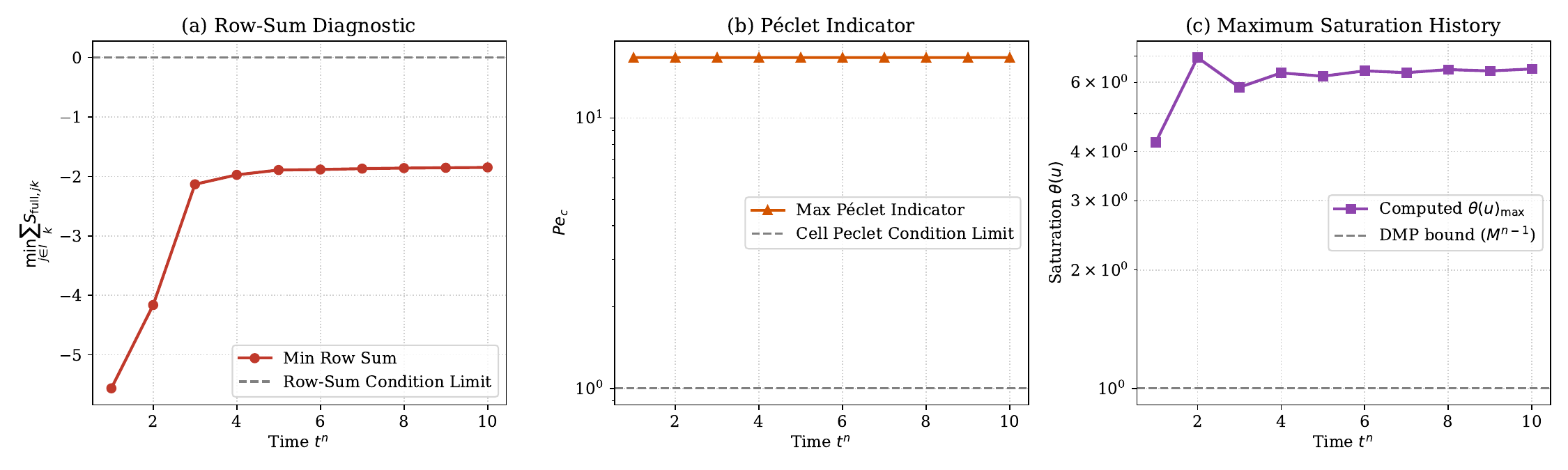}
    \caption{Simulation of an advection-dominated regime ($\alpha=1.0$, $K_s=10$). (a) Minimal row-sum across the interval. (b) Maximal Péclet indicator. (c) Evolution of the maximal computed saturation and the upper bound.}
    \label{fig:overshoot-plot}
\end{figure}

For this numerical simulation, the recorded extrema over $t\in[1,10]$ are $\theta(u)_{\min}=-1.509$ and $\theta(u)_{\max}=6.937$.
Both limits ($0$ and $1$) are violated in this advection-dominated test.

To test whether reducing $h$ can recover the DMP, we ran a mesh-refinement sweep with the same parameters over two time steps ($T=2$, $\tau=1$).
We used $\tau=1$ and final time $T=2$ (two time steps).
The imposed initial and boundary conditions are the same as in Test~\ref{sec:test-advection},
$$
\theta(u_0)(z)=
\begin{cases}
1.0, & z>H/2,\\
0.2, & z\le H/2,
\end{cases}
\qquad
\theta(u)(0,t)=0.2,\ \theta(u)(H,t)=1.0.
$$
We define $N$ as the number of mesh points on $(0,H)$.
The tested values are
$N = 40, 80, 160, 400$.
The effective mesh size is
$$
h_{\mathrm{eff}}=\frac{H}{N-1}=\frac{200}{N-1}.
$$
For each run, we verify the saturation extrema and report the maximal Péclet indicator,
$\min\theta(U)$, and $\max\theta(U)$ over $t\in[1,2]$.

Table~\ref{tab:overshoot-short-sweep} reports that the maximal P\'eclet indicator decreases as $h$ decreases. In this two-time-step sweep, the reported values satisfy $\max\theta(U)>1$ for all listed meshes, while the lower bound $\min\theta(U)<0$ is violated until the mesh is refined to $N=400$.

\begin{table}[htbp]
\centering
\caption{Mesh-refinement sweep for the advection-dominated test over two time steps ($T=2$, $\tau=1$); all reported quantities are over $t\in[1,2]$.}
\label{tab:overshoot-short-sweep}
\begin{tabular}{cccccc}
\toprule
Mesh $N$ & $h_{\mathrm{eff}}$ & Max $Pe_T$ & Min $\theta(U)$ & Max $\theta(U)$ & Max principle? \\
\midrule
$40$  & $5.128$ & $16.743$ & $-1.509$ & $6.937$  & No \\
$80$  & $2.532$ & $4.080$  & $-1.493$ & $8.705$  & No \\
$160$ & $1.258$ & $1.007$  & $-1.777$ & $11.070$ & No \\
$400$ & $0.501$ & $0.160$  & $0.200$  & $14.567$ & No \\
\bottomrule
\end{tabular}
\end{table}

This table shows that the maximal Péclet indicator decreases as the mesh is refined, from $16.743$ at $N=40$ to $0.160$ at $N=400$. For all listed meshes, $\max\theta(U)>1$. The lower bound is violated until $N=400$, where the reported values satisfy $\min\theta(U)\ge 0.2$. This suggests that the discrete minimum principle can be recovered by mesh refinement, but the upper bound is not attained in this test even at $N=400$. This confirms that the discrete maximum principle is not satisfied if the row sums are not positive.

\subsection{Discussion}
The numerical experiments presented in this section systematically validate the theoretical predictions across multiple flow regimes and clarify the distinction between sufficient and necessary conditions.

\subsubsection{Explicit gravity method: conditions are sufficient but not necessary}
The explicit gravity scheme (Section~\ref{sec:explicit}) is proven to preserve positivity under assumption~\ref{ass:strict-timestep}, which imposes a critical time-step bound $\tau_{\text{crit}}^{n-1}$ that vanishes near the degenerate limit. Test~\ref{sec:test1} demonstrates an important practical point: this condition is \emph{sufficient but not strictly necessary}. Despite violating the stability bound across all five time steps ($\mu_{\min}^{n-1} < 0$), the saturation remains bounded within $[0, 1]$. This suggests that in certain problem configurations, the explicit scheme can maintain positivity through compensating nonlinear effects, even when the algebraic condition is violated.

However, Test~\ref{sec:test2} shows that we shouldn't rely on this behavior. With identical spatial discretization but different soil parameters and a coarser time step ($\tau = 5.0$ vs.\ $\tau_{\text{crit}} \approx 0.31$), the scheme produces catastrophic negative saturations ($\theta(U)_{\min} \approx -0.374$). Crucially, the explicit verification run using a stable time step ($\tau = 0.25 \le \tau_{\text{crit}}^{n-1}$) confirms that when the theoretical bound is strictly respected, positivity is guaranteed, with computed saturation remaining bounded within $[0.2, 1.0]$. This pair of tests establishes that while the sufficient condition may occasionally be violated without failure, \emph{respecting the condition provides a rigorous guarantee} for practical use.

\subsubsection{Linearly implicit advection removes severe time-step restrictions}
The comparison test (rerunning Test~\ref{sec:test2} with the linearly implicit scheme) demonstrates the primary motivation for the linearly implicit treatment of gravity. Using the identical coarse time step ($\tau = 5.0$) that caused the explicit method to catastrophically violate bounds, the linearly implicit scheme maintains bounded solutions throughout the entire simulation ($\theta(U)_{\min} = 0.2$, $\theta(U)_{\max} = 1.0$). This represents the central practical advantage: the linearly implicit scheme circumvents the explicit method's crippling time-step restriction near degenerate regimes without transitioning to fully implicit, computationally expensive nonlinear solvers. The ability to use 20$\times$ coarser time steps while maintaining bounds represents a significant computational savings.

\subsubsection{Both conditions satisfied: discrete maximum principle holds}
Test~\ref{sec:test-verified} constructs a diffusion-dominated regime ($\alpha = 0.01$, yielding $h_{\text{cap}} = 100 \gg 1$) where both assumption~\ref{ass:local-peclet} and assumption~\ref{ass:dmp-rowsum} are satisfied throughout the simulation. The diagnostics confirm: (a) minimal row-sum strictly positive at all times, (b) maximal cell P\'eclet number well below 1. Under these conditions, Theorem~\ref{thm:maximum} guarantees the discrete maximum principle, and the numerical results confirm this: computed saturations remain exactly within the prescribed boundary values $[0.2, 0.8]$ at all times. This test directly validates the sufficiency of the theoretical conditions in the favorable regime where both hold.

\subsubsection{Row-sum condition: sufficient but not necessary}
Test~\ref{sec:test-violation} deliberately constructs a regime where the row-sum condition is violated (minimum row-sum $< 0$ due to divergent advective fluxes near the wetting front). Despite this violation, both the minimal and maximal saturations remain within the physical bounds $[0, 1]$. Like Test~\ref{sec:test1} for the explicit scheme, this demonstrates that the row-sum condition is \emph{sufficient but not strictly necessary} for physical bounds to hold in practice. Nonlinear stability and geometric features of the problem can compensate for violations of the algebraic conditions.

\subsubsection{P\'eclet condition and spatial resolution}
Test~\ref{sec:test-advection} explores an extreme advection-dominated regime ($\alpha = 1.0$, $K_s = 10$) with massive cell P\'eclet numbers (up to 16.7) that drastically violate the local P\'eclet condition. The resulting computed solutions violate both bounds severely ($\theta_{\min} \approx -1.51$, $\theta_{\max} \approx 6.94$). The mesh-refinement sweep (Table~\ref{tab:overshoot-short-sweep}) reveals a critical asymmetry:
\begin{enumerate}
    \item \textbf{Lower bound recovery:} As the mesh is refined from $N=40$ to $N=400$, the maximum P\'eclet number decreases from 16.7 to 0.16. Simultaneously, the minimum saturation recovers from $-1.51$ to $0.2$ (within the prescribed bound). This demonstrates that the P\'eclet condition violation is primarily a spatial resolution issue---sufficient refinement can recover the lower bound.
    \item \textbf{Upper bound persistence:} Even at the finest resolution ($N=400$), the maximum saturation remains severely violated ($\theta_{\max} \approx 14.6 \gg 1$). This upper bound violation persists despite the P\'eclet number being driven below the theoretical threshold, indicating that the \emph{row-sum condition is a separate geometric constraint} that cannot be recovered purely through mesh refinement.
\end{enumerate}
This test clarifies the distinct roles of the two conditions: the P\'eclet condition governs convective stability (requiring spatial resolution), while the row-sum condition governs geometric constraint (requiring appropriate mesh geometry or problem structure).

\subsubsection{Summary of theoretical vs.\ practical conditions}
The test suite reveals an important distinction between theory and practice:
\begin{itemize}
    \item \textbf{Sufficiency:} All theoretical conditions (explicit time-step bound, P\'eclet condition, row-sum condition) are proven to be \textbf{sufficient}. When satisfied, physical bounds are guaranteed.
    \item \textbf{Necessity:} Individual conditions are \textbf{not necessary}. Tests demonstrate that bounds can be maintained in specific problem configurations even when conditions are violated.
    \item \textbf{Robustness:} When any single condition is violated, the method may still work on particular problems but offers no guarantee. For such cases, further refinement or problem-specific analysis is warranted.
\end{itemize}

\section*{Acknowledgments}

This work was supported by an NSERC, Canada Discovery Grant (RGPIN-2019-06855) to Yves Bourgault 
and an NSERC, Canada Discovery Grant (RGPIN/5220-2022 \& DGECR/526-2022) to Abdelaziz Beljadid.

\bibliographystyle{siamplain}
\bibliography{references}

\end{document}